\documentclass[a4paper,11pt]{amsart}

\usepackage{amssymb,amsbsy,amsmath,amsfonts,amssymb,amscd}
\usepackage{latexsym}
\usepackage{graphics}
\usepackage{color}
\usepackage{comment}
\usepackage{pdfpages}
\input xy
\xyoption{all}

\theoremstyle{plain}
\newtheorem{thm}{Theorem}[section]

\newtheorem{lemma}[thm]{Lemma}

\theoremstyle{definition}
\newtheorem{remark}[thm]{Remark}

\newtheorem{defin}[thm]{Definition}

\newtheorem{example}[thm]{Example}

\numberwithin{equation}{section}
\newcommand{\sK}{{\mathcal K}}
\newcommand{\sL}{{\mathcal L}}

\newcommand{\PP}{\ensuremath{\mathbb{P}}}
\newcommand\LL{{\mathbb L}}

\newcommand{\CC}{\ensuremath{\mathbb{C}}}
\newcommand{\RR}{\ensuremath{\mathbb{R}}}
\newcommand{\ZZ}{\ensuremath{\mathbb{Z}}}

\newcommand{\hol}{\ensuremath{\mathcal{O}}}
\newenvironment{dedication}
        {\begin{quotation}\begin{center}\begin{em}}
        {\par\end{em}\end{center}\end{quotation}}
\newcommand\la{\lambda}
\newcommand\s{\sigma}

\newcommand\al{\alpha}
\newcommand\be{\beta}

\newcommand\ga{\gamma}
\newcommand\de{\delta}
\newcommand\e{\epsilon}

\newcommand{\ra}{\ensuremath{\rightarrow}}

\def\eea{\end{eqnarray*}}
\def\bea{\begin{eqnarray*}}

\newcommand\dual{\mathrel{\raise3pt\hbox{$\underline{\mathrm{\thinspace d
\thinspace}}$}}}
\newcommand\qe{\ifhmode\unskip\nobreak\fi\quad $\Box$}       % box for QED

\def\BOX{\hfill\lower.5\baselineskip\hbox{$\Box$}}
\theoremstyle{plain}

\newtheorem{theo}[thm]{Theorem}

\newtheorem{remarkk}[thm]{Remark}
\newenvironment{rem}{\begin{remarkk}\rm}{\end{remarkk}}

\newtheorem{prop}[thm] {Proposition}

\usepackage{hyperref}
\usepackage{tikz}

\usetikzlibrary{decorations.pathreplacing}
\tikzset{%
  show curve controls/.style={
    postaction={
      decoration={
        show path construction,
        curveto code={
          \draw [blue] 
            (\tikzinputsegmentfirst) -- (\tikzinputsegmentsupporta)
            (\tikzinputsegmentlast) -- (\tikzinputsegmentsupportb);
          \fill [red, opacity=0.5] 
            (\tikzinputsegmentsupporta) circle [radius=.5ex]
            (\tikzinputsegmentsupportb) circle [radius=.5ex];
        }
      },
      decorate
}}}

\title [Real Abelian coverings]{ Cyclic and Abelian coverings of real varieties}

\author{Fabrizio Catanese}
\address{Lehrstuhl Mathematik VIII, 
 Mathematisches Institut der Universit\"{a}t
Bayreuth, NW II\\ Universit\"{a}tsstr. 30,
95447 Bayreuth, Germany \\ and Korea Institute for Advanced Study, Hoegiro 87, Seoul, 
133--722.}
\email{Fabrizio.Catanese@uni-bayreuth.de}
\author{ Michael L\"onne}
\address {Lehrstuhl Mathematik VIII\\
Mathematisches Institut der Universit\"at Bayreuth\\
NW II,  Universit\"atsstr. 30\\
95447 Bayreuth}
\email{Michael.Loenne@uni-bayreuth.de \newline \hspace*{2.2cm} }
\author{Fabio Perroni}
\address{Dipartimento di Matematica e Geoscienze, Universit\'a degli Studi di Trieste, via Valerio 12/1, 34127 Trieste, Italy}
\email{fperroni@units.it \newline \hspace*{2.2cm} }
\thanks{AMS Classification: 14P99, 14P25, 14H30, 30F50, 12D99. \\
Keywords: Abelian coverings, Real varieties.\\
The present work took place in the framework of the 
 ERC Advanced grant n. 340258, `TADMICAMT'. Very preliminary  results were announced at the Conference 
 `Real Algebraic Geometry', Universit\'e de Rennes 1, June 2011. }

\date{\today}

\begin{document}

\maketitle

\begin{dedication}
Dedicated to Slava (Viatcheslav) Kharlamov  on the occasion of his 71-st  
 birthday.
\end{dedication}

\begin{abstract}
We describe the  birational and the biregular  theory of cyclic and Abelian coverings between  real varieties.

\end{abstract}

\tableofcontents

\section*{Introduction}

While real solutions of polynomial equations were ever since investigated (for instance Newton classified all the
possible equations and shapes of real plane cubic curves) a breakthrough came  alongside of  the impetuous 
 development  of complex function theory.

Harnack, Klein and Weichold, just to name a few, \cite{harnack}, \cite{klein},  \cite{klein1},  \cite{klein3},
 \cite{weichold} used the ideas of Riemann in order to study real equations and their real solutions.
 
 In this way the main branch of real algebraic geometry was born, the one which focuses on the 
 pair of sets given by the complex solutions and the real solutions, with complex conjugation $\s$ acting on them.
 
 The abstract formal outcome of this approach is the definition of a real manifold as a pair  $(X, \s)$ where  $X$ is a complex 
 manifold  and  $\s : X \ra X$   is   an antiholomorphic involution (involution means that  $\s^2$ is the identity).
 The substantial outcome was the use of topological methods and of methods from complex manifolds theory for 
  real algebraic geometry.
 
 A strong boost towards the development of real algebraic geometry came from the Hilbert problems 16 and 17, 
 posed in Paris in the year 1900.
 
 Especially problem 16 oriented the research towards the extrinsic geometry: real plane curves were the object of intensive
 investigations, and, later on,  also surfaces of low degree   in $\PP^3$  (\cite{slava1,slava2}).
 
 The Russian school was very influenced by Hilbert's problem 16, with  an even more prominent  role of topology in the study of this extrinsic geometry of real varieties (see \cite{kharlamov} for an excellent survey): but soon emerged the role
 of intrinsic geometry.
 
 These methods  had been  pioneered by Comessatti, who studied the topology of rational surfaces \cite{comessatti0},\cite{comessatti1},
 and later on the real structures on Abelian varieties  \cite{comessatti2}, complementing previous  results of Klein and
 others on real algebraic curves  (see \cite{cilped} for an excellent survey).
 
 An important endeavour was the extension of the Castelnuovo-Enriques classification of surfaces to the case
  of real algebraic surfaces, with initial contributions in \cite{Manin66}, \cite{Manin67}, 
\cite{Iskovskih67}, \cite{Iskovskih70}, \cite{Iskovskih79}, \cite{Nikulin79},  \cite{silhol}, and then achieved  through  a long series of papers among which
  \cite{RealEnriques}, \cite{CatFred}, \cite{frediani}, \cite{Elliptic} (we refer to \cite{mangolte}, \cite{mangolte2} as an excellent
  text and for more references of many other works by Mangolte and others).

  This said, the contribution of the present paper concerns the intrinsic real algebraic geometry 
  and deals with Abelian coverings between real algebraic varieties $(X, \s)$ and $(X', \s')$ (i.e., Galois coverings such that the Galois group of $ X \ra X'$ is an Abelian group $G$ normalized by $\s$).
  The complex case was initiated by Comessatti \cite{comessatti} who described  the cyclic case 
  and, after many intermediate works by several authors, the biregular theory
  of such coverings, in the case where $X'$ is smooth, was established  in \cite{pardini}. 
  
  The present paper is divided into two parts: the first one is devoted to the birational theory, i.e., the description 
  of the corresponding fields extensions $\RR(X') \subset \RR(X)$, the second is the biregular theory for $X'$ factorial,
  where the covering is described  according to the scheme of  Pardini's paper through building data, which are divisors on $X'$ and line bundles on $X'$  satisfying some compatibility relations, and where we have to add the reality constraints.

  We should warn the reader that the birational theory, which is well-known and almost trivial in the complex case
  (where the cyclic case is the one where the field extension $\CC(X') \subset \CC(X)$
  is obtained by taking the $n$-th root $z$ of a function $f \in \CC(X')$, hence it is  described by the simple formula $z^n = f$),
  is by no means easy in the real case.

  This is the reason why the  title distinguishes between the cyclic and the Abelian case: 
  while we can describe the cyclic case as the fibre product 
  (field compositum) of four basic fields extensions, this becomes considerably more complicated in the Abelian case.
  The underlying reason is a simple arithmetical fact: the group $G'$ generated by the Galois group $G$ and by
  the antiholomorphic involution $\s$ is a semidirect product $G' = G \rtimes (\ZZ/2)$, classified by an automorphism $M : G \ra G$
  such that $M^2=1$ ($1$ is here  the identity of $G$).
   We dedicate Section 2 to the analysis of the algebraic structure of the group $G'$.
  
  In the cyclic case $M \in ( \ZZ/n)^*$ can be described through a splitting of the cyclic group $G$ as 
  a direct sum such that, on each summand  $(\ZZ/h)$, $M$ acts either as multiplication by $\pm 1$  or, possibly,  if $h$ is divisible by $8$,
  by $\pm 1 + \frac{h}{2}$.

  In the non cyclic case there is no such similar  splitting and already the description of such pairs $(G,M)$
  is slightly more complicated (see Lemma \ref{M}).
  
   Section 3 is dedicated to the description of the corresponding fields extensions.

 The easiest case where the field extension is given by $z^n = f \in \RR(X')$ is the one where the group $G'$ is 
 isomorphic to the dihedral group $D_n$  (and the field extension $\RR(X') \subset \RR(X)$ is not Galois). The other cases become progressively more complicated,  and the main result
 of Section 3, and of the first part,  is Theorem \ref{field}, describing the four basic cases. 
 Since the statement of the theorem requires a complicated notation, we do not follow  the usual 
 order of exposition: instead, while preparing the necessary definitions,  we provide at the same time the proof.  
 
Section 4 is devoted to the description of the  biregular theory of real Abelian coverings in terms of 
branch divisors and  invertible character sheaves, and the results are spelled out in particular in Theorem
\ref{cyclic} (recalling the cyclic case in the complex setting), in Theorem \ref{real-cyclic}, stating the result
in the real cyclic case, while for the general Abelian case, to avoid too much repetition, we just indicate 
how to get the result from Theorem \ref{real-cyclic} mutatis mutandis.

   \bigskip
   
  Passing to the biregular theory  of more general Galois coverings, with non Abelian Galois group $G$,
   we should recall the real analogue of the complex Riemann existence theorem 
  (in the general version given by Grauert and Remmert \cite{G-R}).
  
  It is based on the notion of the real fundamental group of $(X, \s)$, denoted  $\pi_1^{\RR}(X, \s)$, which,
  in the case where $X$ has  a real point $x_0 \in X(\RR)$,
 is the semidirect product  
  $\pi_1(X, x_0) \rtimes (\ZZ/2)$, where conjugation by $(\ZZ/2)$ is given by $\s_*$;   else, it is defined as the fundamental group of the Klein variety $\sK(X) : = X / \s$ (see \cite{CatFred}).

 In both cases we have an exact sequence:
 $$ 1 \ra  \pi_1(X) \ra \pi_1^{\RR}(X) \ra  \ZZ/2 \ra 1,$$
 which splits in the first case.

 Putting together definitions and  results of \cite{G-R} and of \cite{CatFred} we obtain the following theorem,
whose general statement might  be new.
 
\begin{theo}\label{RRiemann} (Real Riemann existence theorem)
A real Galois covering with Galois group $G$ between normal real varieties $(X, \s)$ and $(X', \s')$
is determined by  a $\s'$-invariant Zariski closed subset $B$ of $X'$ and a surjective homomorphism 
$$
\Psi : \pi_1^{\RR}(X' \setminus B, \s') \ra G' \, ,
$$
where $G' = G \rtimes (\ZZ/2)$  is determined by an automorphism $M \in Aut(G)$ of order $2$ (i.e. $M^2=1$),
such that the following conditions are satisfied:  
\begin{itemize}
\item[i)]  
the action of $G'$ on $X$, which is determined by $\Psi$, is such that its restriction to 
$\ZZ/2 < G'$   coincides with the action associated to the real structure $\sigma$ on $X$;
\item[ii)] 
$ \Psi (\pi_1(X' \setminus B)) = G$.
\end{itemize}  
 \end{theo} 
 
 In the case where $X'$ is smooth, there is a minimal $B$ which is a divisor. The analysis of  its components
 and of  the local monodromies in the Abelian case leads to the biregular theory of such coverings,
 to which section $4$ is devoted.
 
 In general, the Riemann existence theorem works fine with complex curves, where the fundamental groups
  $\pi_1(X' \setminus B)$ are well known. 
  
In a sequel to this paper we shall treat the case of real curves, which  goes back to the Klein theory 
(revived and extended for instance in \cite{a-g}, and many other papers by Sepp\"ala and other authors, 
for instance  \cite{seppala}): here many arguments, except the final theorems, work for any finite group $G$.  
We shall discuss  such coverings in terms of certain numerical
and topological data, extending  results of \cite{cyclic} in the complex case  (this was partly done in \cite{nankai} and \cite{cime}), and with a view to the study 
of the corresponding moduli spaces of real curves with Abelian symmetry group
$G$  and with a fixed set of  invariants ($M$, local monodromies, ...).

\bigskip

\section{The basic set-up}

In this paper, we consider the following situation:

1) $(X, \s)$ is a real projective variety (this  means that $X$ is a complex projective variety given together with an antiholomorphic
self map $\s : X \ra X$ which is an involution, i.e.  $\s^2$ is the identity);

2) $G $ is a finite subgroup of the group of complex automorphisms of $X$, and $G$ is normalized by $\s$.
If we consider the quotient $X' : = X / G$, since $\s G = G \s$, we see 
 that $\s$ induces an antiholomorphic involution $\s'$ on $X'$, defined by 
 $$\s' (Gx) : = \s (Gx) = G \s x.$$
 
 Hence, 
in particular, $(X', \s')$ is also a real projective variety.
\smallskip

 We can of course consider the more general situation where $X$ is a real space (replace the condition that
$X$ is a projective variety by the condition that $X$ is a complex space). 

\begin{defin}
Saying that  $ f : (X, \s) \ra (X', \s')$  is real $G$-Galois covering means that $ G \subset Aut(X)$ as in 2) above,
$X' \cong X/G$ via $f$, and $f$ is real, that is, $f \circ \s = \s' \circ f$.

In this situation $\s$ normalizes $G$, and we have a semidirect product 
$$ G' : = G \rtimes \ZZ/2, \ \ZZ/2 \cong \{ 1, \s\}.$$
\end{defin}

If the finite group $G$ is Abelian, respectively cyclic,  we shall say that  $X \ra X'$ is an Abelian,
respectively  cyclic, covering of real varieties, with group $G$.

For a projective variety $X \subset \PP^N (\CC)$ defined by polynomial equations with real coefficients, 
$\s$ is induced by complex conjugation 
$$\s (x_0, \dots, x_N) : = (\overline{x_0}, \dots, \overline{x_N}).$$
It is important to observe that any real projective variety admits such an embedding,  see for
instance  \cite{mangolte}, Th\'eor\`eme 2.6.44.

More generally, if $X$ is a complex variety, one says that $\s$ is antiholomorphic if it is differentiable
and  locally induced by an antiholomorphic map
of complex manifolds, i.e., such that the derivative $D \s$ of $\s$ satisfies 
$$ J \circ  D \s = -  D \s  \circ J,$$
where $J$ is the complex structure on the (Zariski) tangent space of $X$ ($J^2 = -1$).

The group  $G$ acts on the function field $\CC(X)$, and the antiholomorphic map $\s$ induces also an involution
$\tau : \CC(X) \ra \CC(X)$, which normalizes the action of $G$  on $\CC(X)$,  given by $ g (f) : = f \circ g^{-1}$.

It is easy to understand $\tau$ in the prototype case where $X$ is a projective variety $X \subset \PP^N (\CC)$
defined by polynomial equations with real coefficients:
$\tau$ is induced by the map that acts  on homogeneous polynomials $  P (x_0, \dots, x_N) = \sum_I a_I x^I$
just by conjugating their coefficients.

I.e., $\tau (P) : = \sum_I \overline{a_I } x^I$; and the field of rational functions which are $\tau$-invariant is just the field
$\RR(X)$, such that $\CC(X) = \RR(X) \otimes_{\RR} \CC$.

 In this representation as a tensor product, $\tau$ is induced by complex conjugation
on $\CC$.

In the general case, since $$\overline { P (\s (x))} = \overline { \sum_I a_I \overline{x^I}} = \sum_I \overline{a_I } x^I,$$
one defines 
$$\tau (f) : = \overline { f \circ \s}.$$

We shall work, at least in the first part of the paper, with the group $G'$ generated by $G$ and $\tau$ acting on the function field
$\CC(X)$. The same calculation can be performed in the general case working with line bundles instead 
of function fields.

 In the case where the group $G$ is cyclic, $G \cong \ZZ/n$, since the field $\CC$ contains all roots of unity, the field extension 
$\CC(X) \supset \CC(X')$
has a simple description, as $$ \CC(X) = \CC(X')[z]/ (z^n - f).$$

But in this representation we do not see the action of $\tau$. Indeed, in the next section, we show that, through the fibre product of
two distinct coverings of $X'$, we can reduce to four basic cases:

\begin{enumerate}
\item
$G'$ is the dihedral group $D_n$: this is the standard  totally real covering case where $$ \RR(X) = \RR(X')[z]/ (z^n - f).$$
\item
$G'$ is a direct product $(\ZZ/n) \times (\ZZ/2)$: this is a case which resembles the one of the complex dihedral coverings 
(see \cite{cat-per}), so we call it the dihedral-like case.

  It is important to observe that this is exactly the case where
the field extension $ \RR(X) \supset  \RR(X')$ is Galois (with group $G$):  since the subgroup $\ZZ/2$ is a normal subgroup of $G'$
if and only if $G ' \cong G \times \ZZ/2$.
\item
$n$ is divisible by $8$ and $G'$ is the semidirect product such that 
$$ \tau g \tau = g^{1 + n/2} :$$ this will be called  the twisted case,  and it can be reduced to the dihedral-like case.
\item
$n$ is divisible by $8$ and $G'$ is the semidirect product such that 
$$ \tau g \tau = g^{-1 + n/2} :$$ this will be called the esoteric covering \footnote{It appears esoteric only because we wish to find explicit algebraic formulae for the field extension, instead of using the real version, Theorem \ref{RRiemann}, of  Riemann's  existence theorem of \cite{G-R}}.

\end{enumerate}

\section{A numerical lemma}

In  the first part of  this section we consider the following situation: we are given a finite group $G'$, which is a semidirect product
$$ G ' : = (\ZZ/n) \rtimes  (\ZZ/2) = : G  \rtimes  (\ZZ/2) \ , G : = \langle g |g^n = 1\rangle, \ (\ZZ/2) : = \langle \tau  |\tau^2 = 1\rangle.$$

The semidirect product is then classified by an element $ m \in (\ZZ/n)^*$, such that $m^2 = 1$  and such that 
$\tau g \tau = g^m$.

The next lemma shows that, up to  letting  $G$  be a direct product of two cyclic groups (of coprime order), we can reduce to four basic situations: 

$m= \pm 1$, 
or, when  $n$ is divisible by $8$, $ m=\pm 1 + n/2$.

\begin{lemma}\label{squareroot}
Let $ m \in (\ZZ/n)^*$ be such that $m^2 = 1$. Then we can write $n = n_1 \cdot n_2 \cdot 2^k$, where:
\begin{enumerate}
\item
 $k=0$ or $k\geq 3$, 
 \item
 the factors are pairwise relatively prime (and possibly equal to $1$),
 \item
  using the Chinese remainder theorem to identify
$$\ZZ/n \cong  \ZZ/n_1 \times \ZZ/n_2 \times \ZZ/2^k,$$
$m$ corresponds to $(1, - 1, \pm 1 + 2^{k-1})$. 
\end{enumerate}
In particular, we can write $n = N_1 \cdot N_2$ with $N_1$, $N_2$ relatively prime (and possibly equal to $1$), and either
\begin{itemize}
\item[(4)] 
$ m \in  \ZZ/N_1 \times \ZZ/N_2$ equals $(1, -1)$, or 
\item[(5)] 
$N_2$ is divisible by $8$ and $ m \in  \ZZ/N_1 \times \ZZ/N_2$ equals $(1, -1 + N_2/2)$, or 
\item[(6)]
$N_2$ is divisible by $8$ and $ m \in  \ZZ/N_1 \times \ZZ/N_2$ equals $(-1, 1 + N_2/2)$. 
\end{itemize}

\end{lemma}
\begin{proof}

Let us first consider the case where $n$ is a prime power and  show that (the case $n=2$ being trivial):
\begin{itemize}
 
\item[(i)]
if $n = p^e$,  where $p$ is an odd prime, then $m= \pm 1$;
\item[(ii)]
if $n = 2^e$ and $e \geq 2$,  then either $m= \pm 1$ or,   in case  $ e \geq 3$, 
we can also have  $ m=\pm 1 + n/2$.
\end{itemize}
Indeed in both cases, (i) and (ii), $m \equiv \e  \ ( mod \ p)$, where $\e = \pm 1$. 
In case (i), let $p^h$ the biggest power of $p$ that divides $m - \e$ ($h \geq 1$). 
Then, $1= m^2= ( \e + a p^h)^2 = 1 + 2 \e a p^h + a^2 p^{2h}$, hence $h =e$ since $p > 2$.

In case (ii), it follows that $h \geq e-1$, hence either $h=e$ and $ m = \e$, or $h=e-1$ and $m = \e + 2^{e-1}$.

In the general case where $n>1$ is any natural number, it suffices to use the primary decomposition of $\ZZ/n$
and put together all the primary summands where $m$ equals respectively $1$, $-1$, $ \pm 1 + 2^{k-1}$.
Here $k=e$, if $n=n_1 \cdot n_2 \cdot 2^e$ with $e\geq 3$ and the summand of $m$ belonging to 
$\ZZ/2^e$ is equal to $\pm 1 + 2^{e-1}$ (notice that, if $e=2$, $\pm 1 +2 = \mp 1$). This shows (1), (2) and (3).

Finally, to prove (4), (5) and (6) observe that, 
if the third summand is nontrivial and $m= ( 1, -1, 1 + 2^{k-1})$, then, setting $N_2 : = n_1 \cdot 2^k$, 
we have that  $ 1 + n_1 \cdot 2^{k-1}$ is congruent to $1$ modulo $n_1$
and congruent to $ 1 +  2^{k-1}$ modulo  $2^k$. The other case where $m= ( 1, -1, -1 + 2^{k-1})$,
that is,  $ m = -1 + n_2 2^{k-1}$, 
is entirely similar.

\end{proof}

 The previous lemma says therefore that there are seven possibilities:  
three contemplated in (4) (namely $N_1=1$, or $N_2=1$, or $N_1 \neq 1 \neq N_2$)  and respectively two for both 
(5) and (6) (namely, $N_1 =1$, or $N_1 \neq 1$, since here we must have $N_2 \neq 1$). 
Of these  four are `pure' (unmixed cases), and three are direct products of two pure cases.

The way the lemma shall be applied is through the following well-known proposition; when stating our results we shall sometimes
take this condition for granted (it distinguishes between irreducible and reducible covers).

\begin{prop}\label{nonpower}
Let $K$ be a field of characteristic zero and containing all roots of $1$. 

Then $ L :=  K [z] / (z^n -f) $ is a field if and only if there is no divisor $h$ of $n$, with $ h \geq 2$, and no element $a$ of $K$ such 
that $ a^h = f$ (equivalently, the same condition with $h$ a prime number).

Assume now that $ L :=  K [z] / (z^n -f) $ and $ L' :=  K [x] / (x^m -\phi ) $ are fields, and that $n,m$ are relatively prime.
Then $L \otimes_K L'$ is a field extension of $K$, Galois with Galois group cyclic of order $nm$.
\end{prop}

\begin{proof}
For the first assertion, assume that such an element $a$ exists and set $n = hk$: then $z^n - f = z^{hk} - a^h$,
which is divisible by $z^k - a$.

Conversely, assume that $F :  = z^n - f $ is not irreducible in $K[z]$. We have an action of the group $G$
of $n$-th roots of $1$ on $K[z]$, $ z \mapsto \zeta^j z$. Since $F$ is $G$-invariant, $G$ permutes the
irreducible factors of $F$. If $F$ has a linear factor $z-a$, we have that $a^n = f$. Otherwise, the orbit of
a factor $P$ of minimal degree has cardinality $ < n$, hence we have a factor $P$ stabilized by a subgroup of $G$, generated by some power 
$\zeta^h$, where $ h \geq 2$ divides $n$,
and equal to the group of $k$-th roots of $1$ ($n = kh$).

Therefore this factor is of the form 
$$ P(z) = \sum_j b_j z ^{kj} ,$$
and, setting $ w : = z^k$, we get that $ Q(w) : = \sum_j b_j w ^{j} $ divides $w^h - f$.

Now, $Q$ must be linear, otherwise it would have a nontrivial stabilizer in the group of $h$-th roots of $1$,
contradicting that the stabilizer of $P$ has order exactly $k$. Since $Q$ is linear, it is of the form $w-a$,
hence $ a^h = f$.

For the second assertion, observe that $$L \otimes_K L' = K[x, z] / (z^n - f, x^m - \phi) $$ admits an action by the group
$\ZZ / n \times \ZZ/m = \ZZ/nm$ (since $n, m$ are relatively prime) which, given $\zeta$ a primitive $n$-th root of $1$, and $\e$  a primitive $m$-th root of $1$,
sends $ z \mapsto \zeta^i z$, $ x  \mapsto \e^j x$. Hence $xz$ is an eigenvector for  the standard  generator of $\ZZ/nm$, in particular it generates 
the $K$-algebra  $L \otimes_K L' $, satisfying
 indeed the relation $ (zx)^{nm} = f^m \phi^n$.

By the first assertion,
if $L \otimes_K L' $ is not a field, there exists $a \in K$ such that $a^h = f^m \phi^n$, where $h \geq 2$ divides $nm$.

Without loss of generality we may assume that $h$ is prime, and that it divides $m$, so that $ m = hk$.

Since $a^h = f^{hk} \phi^n$, it follows that, setting $ b : = a \cdot f^{-k}$, $b^h = \phi^n$.

However, since $n,m$ are relatively prime, there exist integers $r,s$ such that $ 1 = hr + ns$.

We finally derive:
$$ \phi = \phi ^{hr + ns} = ( \phi^r  b^s  )^h = : c^h$$
and since $h | m$ , we have contradicted that $L'$ is a field, because of  the first assertion of this proposition.

\end{proof} 

\begin{rem}\label{fact}
In the case where $K = \CC(X')$ is the function field of a factorial variety $X'$, the condition of existence of $a$
with $f = a^h$ can be verified once we write   $f= \frac{s}{t}$ as the quotient of two 
relatively prime sections of a line bundle $\sL$, and we then take the unique factorization of $f$:
$$f = \frac{\Pi_i s_i^{n_i}}{\Pi_j t_j^{m_j}},$$
where the $s_i, t_j,$ are prime.

$h$ should divide the greatest common divisor $GCD(n_i, m_j)$ of all the exponents $n_i, m_j $.
We shall often omit to specify this condition each time.

\end{rem}

More generally, we can consider the situation where $G$ is a finite 
Abelian group and we have a semidirect 
product $$G '  = G \rtimes (\ZZ/2),$$ where the semidirect product is determined by an automorphism
$$ M \in Aut (G), \ {\rm such \ that } \ M^2 = 1 : = Id_G .$$

Using the primary decomposition of $G$, $ G = \oplus_p G_p$, where $G_p$ is the $p$-primary component,
it suffices to describe $M_p : G_p \ra G_p$. 

$G_p$ is isomorphic to a direct sum 
$$G_p = \oplus_r (\ZZ/(p^r))^{n_r} \, .$$

\begin{lemma}
Let $G_p = \oplus_r (\ZZ/(p^r))^{n_r}$ be a finite Abelian group of exponent 
 a power of $p$, where $p$ is an odd prime,
and $M_p : G_p \ra G_p$ with $M_p^2 = 1$:
then we have a splitting $G_p = G_p^+ \oplus G_p^-$ into $+1$ and $-1$ eigenspaces. 
\end{lemma}
\begin{proof}
As usual, $G_p^+ = \{ \frac{1}{2} (x + M_px) \} , \ G_p^- = \{ \frac{1}{2} (x - 
M_px)\} .$

\end{proof}

The following Lemma applies in particular to the case of a finite Abelian group of exponent $2^s$, 
$G_2 = \oplus_r (\ZZ/(2^r))^{n_r}$, endowed with an automorphism $M_2 : G_2 \ra G_2$ with $M_2^2 = 1$:
 
\begin{lemma}\label{M}
Let $G$ be a finite Abelian group  endowed with an automorphism
 $M : G \ra G$ such that  $M^2 = 1$.
 
Let $$G^+ : = Ker (M - 1), G^-  : = Ker (M + 1)$$ be the respective eigenspaces for the eigenvalues $+1, -1$, and 
define $U : = G^+ \cap G^- $, $W = G^+ + G^-$.

Then 

(1) $U = G^+_{[2]}=  G^-_{[2]}$, where $ G^{\pm}_{[2]}$ denotes the subgroup of $2$-torsion elements
of $G^{\pm}$, and we have a canonical isomorphism $W \cong  (G^+ \oplus G^- )/ U$;

(2) $ V : = G/W$ is a vector space  over $\ZZ/2$,  with an induced action by $M$, which is the identity;

(3)  for $x \in G, M (x) = x + F(x)$, where $F : G \ra G^-$ vanishes on $G^+$, and equals $-2x$ for $ x \in G^-$.

(4) For fixed $G,  G^+, G^-$, we have a set of possible homomorphisms $M$ with $M^2=1$ and such that  
$G^+$ is contained in the $(+1)$-eigenspace 
for $M$, and $G^-$ is contained in the $(-1)$-eigenspace for $M$,
parametrized by  the set 
$ Hom (V, U)$. 

(5) Take now an Abelian group $G$ containing  subgroups $W, G^-, G^+$ such that 
\begin{enumerate}
\item[(I)]
$W  = G^- +  G^+$,  
\item[(II)]
$W \cong (G^- \oplus  G^+)/U$,
where $  U = G^+_{[2]}=  G^-_{[2]},$
\item[(III)]
$G/W = : V$ is a  vector space  over $\ZZ/2$.
\end{enumerate}
  
Let $ F_W : W \ra G^-$ be defined as in (3).

 Then there exists an extension $F : G \ra G^-$ of $ F_W$, and  the set of automorphisms of $G$ with $M^2 =1$
 such that   $G^+$ equals the $(+1)$-eigenspace
  for $M$, $G^-$ equals the $(-1)$-eigenspace for $M$,
 is in bijection with  all those  extensions $F' = F + \phi, \phi \in Hom (V, U)$, satisfying the further properties:
 
 $G^+ = Ker (F')$,  $G^- = Ker (F' + 2)$.
 
\end{lemma}

\begin{proof}
Let us first consider the baby case where   $G$ is a vector space  over $\ZZ/2$.

Since  $ M : G \ra G$ satisfies $M^2=1$,  we get a Jordan normal form for $G$ with 
eigenvalues $1$ and  blocks of length $\leq 2$, that is, there is a direct sum $ G = G^+ \oplus G' $
such that $ M(v^+, v') = (v^+ + f(v'), v' )$, with $ f : G' \ra G^+$ linear and injective. Observing that $G' \cong G/G^+$,
we have established  that 
 we have  an injective  homomorphism
$ f : G/G^+ \ra G^+$ such, that for $y \in G$, letting $p_1 : G \ra G/G^+ $ be the natural surjection, then $M(y) = y + f (p_1 (y))$.
\medskip

In general, since $M : G \ra G$ satisfies  $M^2 = 1$, we get two filtrations by  submodules
$$ 0 \subset   Im (M - 1) \subset   Ker (M + 1) =:  G^- \subset  G = Ker (M^2 - 1),$$
$$ 0 \subset   Im (M + 1) \subset Ker (M - 1)  =:  G^+ \subset G = Ker (M^2 - 1) ,$$
which are invariant since $M$ acts as $-1$  on $G^-$, and as $+1$ on $G^+$.

Observe then that $U : = G^+ \cap G^-  = \{ y| 2 y =0, My = y\} \subset G_{[2]}$, hence (1)  follows immediately.

Moreover, $ 2y = (M + 1) y - ( M - 1)y$ shows that $G / W$ has exponent $2$, 
hence it is a vector space $V$ over $\ZZ/2$.   As we noticed,
$ Im (M - 1) \subset     G^- \subset W$, hence $(M-1) : G/W \ra G/W$ is $=0$,
hence (2) holds.

Write $M(x) = x + F(x)$.

Here $ F = M - 1$ vanishes on $G^+$, and equals $-2x$ for $ x \in G^-$, hence (3) is established
and $F|_W$ is canonically determined.

Now, $F : G \ra G^-$ is an extension of $F|_W$.

Observe that $G$ is an extension:

$$ 0 \ra W \ra G \ra V \ra 0.$$

Hence we have an exact sequence

$$ 0 \ra Hom (V, G^-) \ra Hom (G, G^-) \ra Hom (W, G^-) \ra Ext^1 (V, G^-).$$

Since $V$ has exponent $2$, and  $F|_W$ is  determined, the choice of an extension $F$ of  $F|_W$ 
is  determined up to a
homomorphism  $ \phi \in Hom (V, U) $.

$\phi$ can be taken arbitrarily, since $M$ acts as the identity on $V$, hence also $ (M + \phi)^2 = 1$:
hence (4) holds.

Take now an Abelian group $G$ containing  subgroups $W, G^-, G^+$ as in (5) such that $W  \cong (G^- \oplus  G^+)/U$,
where $  U = G^+_{[2]}=  G^-_{[2]}.$

Take $F|_W$ as above: the question is whether there exists $F : G \ra G^-$ extending $F|_W$.
If the answer is positive, then we define $M : = F +1$, and clearly $M^2=1$.

It suffices to see that $F|W $ is sent to $ 0 \in Ext^1 (V, G^-)$, because then  there exists $F \in Hom (G, G^-)$ 
extending $F|W $.

Now, the   coboundary map of the  exact sequence  is  given  via the natural pairing 
$$ ( Ext^1 (V, W ) \times Hom ( W, G^-)) \ra  Ext^1 (V, G^-) $$ 
by pairing with the extension class of $ 0 \ra W \ra G \ra V \ra 0 $
(an element of $Ext^1 (V, W )$).

Observe also that $Ext^1 (V, W)$, since $V$ is a vector space over $\ZZ/2$, equals 
$$V \otimes_{\ZZ/2} Ext^1 (\ZZ/2, W) = V \otimes_{\ZZ/2} (W/ 2W),$$
similarly $Ext^1(V,G^-)$  equals
$$V \otimes_{\ZZ/2} Ext^1 (\ZZ/2, G^-) = V \otimes_{\ZZ/2} (G^-/ 2G^-),$$

Now, we have that $F|_W $ sends $W \ra 2 G^-$, hence $F|_W$ maps to $0$ and there exists $F \in Hom (G, G^-)$ 
extending $F|_W $.

Finally, $M(x) = x \Leftrightarrow F(x) = 0$, and $M(x) = - x \Leftrightarrow F(x) = - 2x$.

This said, we have that $G^+$ is just contained in the $(+1)$-eigenspace, and similarly  $G^-$ is 
just contained in the $(-1)$-eigenspace. Equality holds if  $G^+ = Ker (F)$, $G^- = Ker (F + 2)$
\end{proof}

\begin{rem}
(i) The case where $G$ is of exponent $2$ (a vector space over $\ZZ/2$), and $ U = W = G^+$ shows that the $2$- torsion subgroup of $G$
can be strictly larger than $U$.

(ii) The same example shows that, fixed $G^+$, the $(+1)$-eigenspace 
for $M$ contains $G^+$ and it is equal
to it if and only if $f : V \ra U$ is injective. If we change  $f$  with $ f + \phi$, and the latter is not injective,
then $(+1)$-Eigenspace becomes larger, it is equal to the inverse image of  $Ker ( f + \phi) \subset V $
under the projection $ p : G \ra V$. Likewise also $U$ becomes larger.

(iii) An easy issue of Lemma \ref{M} 
is the one where  $G$ is any group, and $ M = \pm 1 + \varphi$,
where $\varphi : G \ra G_{[2]}$ is a homomorphism vanishing on $G_{[2]}$. Let $\e \in \{ +1, -1\}$,
so that  $ M (x) = \e x + \varphi(x)$.  Then 
$$ M^2 (x) = \e (\e x + \varphi(x)) + \varphi ( \e x + \varphi(x)) = x +  \varphi( \varphi(x)) = x.$$

In the case $\e = 1$, $G^+ = Ker (\varphi) \supset G_{[2]}$, $G^- = Ker (\varphi -2)$, and 
$$ U = Ker (\varphi) \cap  Ker (2) =G_{[2]}.$$
\end{rem}

\section{The function field extensions in the four basic cases}

We begin with a 
very simple observation.
\begin{remark}\label{realstructures}
Let $\RR(X')$ be a finitely generated extension field of $\RR$ which does not contain $\sqrt{-1}$, let $\Phi \in \RR(X')$ be not a square, and consider the field extension $\RR(X^0)$ given by $y^2 = \Phi$. 

Then there are 
two real structures on $\RR(X^0)$: for one of them (the one we shall choose) 
 $\tau(y) = y$,
for the other $\tau'(y) = - y$ ($\tau ' = \iota \circ \tau$, where $\iota$ is the Galois  involution $y \mapsto - y$).

\end{remark}

\subsection{The case where $G'$ is a dihedral group}

We have here $ \tau g \tau = g^{-1} $.

Let us assume that the cyclic field extension is generated by $w$, such that $w^n = F$, $F \in \CC(X')$.

If $g$ is a generator of $G$, we have $$ g(w) = \zeta w,$$ where $\zeta$ is a primitive $n$-th root of $1$.

Hence
$$ g (\tau (w)) = \tau ( g^{-1} (w)) = \tau (\zeta^{-1} w) = \zeta \tau (w),$$
so that $\tau (w)$ is also an eigenvector for  $g$ with eigenvalue $\zeta$ 
(in particular  there exists $A \in \CC(X')$ with $\tau(w) = A w$).

Let $V_{\zeta}$ be the eigenspace for  the eigenvalue $\zeta$: then any nonzero element $z$ inside $V_{\zeta}$ generates the extension
$ \CC(X') \subset \CC(X)$. 

If there is now a $w \in V_{\zeta}$ such that $z : = \tau(w) + w \neq 0$, then $z\in \RR(X)$ generates the extension 
and $z^n = f \in \RR(X')$, as we want to show.

Otherwise, $\tau$ would act as $-1$ on $V_{\zeta}$. This assumption leads to a contradiction, since
$V_{\zeta}$ is a complex vector space, and we would have $\tau(\la w) = \overline{\la} (-w) \neq - \la w$
if $\la$ is not a real number.

\subsection{The case where $G'$ is a direct product $G' = G \times \langle \tau\rangle$.}

We have here $ \tau g \tau = g $.

A first observation is that $g$ preserves the two eigenspaces for $\tau$, therefore $ g (\RR(X)) = \RR(X)$,
and $ g (\sqrt{-1} \RR(X)) = \sqrt{-1} \RR(X)$.

Assume that the  field extension is generated by $z$, such that $z^n = f$, $f \in \CC(X')$.

If $g$ is a generator of $G$, we have $ g(z) = \zeta z$, where $\zeta$ is a primitive $n$-th root of $1$.

Hence
$$ g (\tau (z)) = \tau ( g (z)) = \tau (\zeta z) = \zeta^{-1} \tau (z),$$
so that $\tau$ exchanges the two eigenspaces $V_{\zeta}$ and $V_{\zeta^{-1}}$.

Hence we have $$\phi : = z \cdot \tau(z) \in \RR(X').$$

Moreover, if $ z = a + \sqrt{-1} b$, 
$a,b \in \RR(X)$, then $\phi = a^2 + b^2$, a  function which on the real locus takes real positive values.

If we now set $ w : = \tau(z) $, we have $w^n = \tau(f)$ and, setting 
$$  F : = f + \tau(f) \Rightarrow F  \in \RR(X'),$$ 
we have
$$ (*) \ z w = \phi , \ z^n + w^n =  F. $$  
It follows that $f , \tau(f)$ are roots of the quadratic equation
$$ x^2 -  F x + \phi^n = 0 .$$

 Let $\psi \in  \RR(X')$ be twice the imaginary part of $f$: then 
$$f = \frac{1}{2}(F + \sqrt{-1} \psi), \tau(f) = \frac{1}{2}(F - \sqrt{-1} \psi )$$ and 
the discriminant $\Delta$ is minus a square in $ \RR(X')$, since
$$ (**) \, \Delta =  F^2 - 4 \phi^n = - \psi^2.$$

Conversely, if there are $F, \phi, \psi \in  \RR(X')$ satisfying $(**)$, then  equations $(*)$  define
 a cyclic extension of $\CC(X')$ of order $n$: since we can eliminate $w = \phi / z$, and for the resulting
equation
$$ z^{2n} - F z^n + \phi^n = 0 \Leftrightarrow (z^n -  \frac{1}{2}(F + \sqrt{-1} \psi)) 
(z^n - \frac{1}{2}(F - \sqrt{-1} \psi ))=0$$
we can take the root  $z$ of 
$z^n = f : = \frac{1}{2}(F + \sqrt{-1} \psi)$  or the root $w$ of  
$w^n = \tau(f) = \frac{1}{2}(F - \sqrt{-1} \psi)$.

 If  $f : = \frac{1}{2}(F + \sqrt{-1} \psi)$  satisfies the property  that there do not  exist any divisor of 
 $n, h>1,$ and a rational function 
$\varphi = \al + \sqrt{-1} \beta$ with 
$(\al + \sqrt{-1} \beta) ^h=f$, then  Proposition 2.2 applies and the cyclic extension is a field
(cf. remark \ref{fact}). 

Exchanging $\psi$ with $- \psi$ has the effect of replacing $z$ with $w$.

We can then extend $\tau$ to $\CC(X)$ by setting $ w : = \tau(z) \Rightarrow z = \tau(w)$.

Finally, concerning the existence of such functions $F, \phi, \psi$ we simply use the rationality of the variety defined by $(**)$.

Namely, let $n = 2m + 1$ for $n$ odd, else let $n = 2m+2$. Setting 
$$ F = \phi^m \hat{F}, \ \psi = \phi^m \hat{\psi},$$
we reduce to the respective equations
$$ \phi = \frac{1}{4} (\hat{F}^2 + \hat{\psi}^2), \  4  \phi^2 = \hat{F}^2 + \hat{\psi}^2.$$

In the first case, where $n = 2m + 1$, we see that the solutions correspond to the choice of two 
arbitrary functions  $\hat{F},   \hat{\psi} \in \RR(X')$
 such that $(\hat{F}^2 + \hat{\psi}^2)^m (\hat{F} + \sqrt{-1} \hat{\psi})$ 
 is not of the form $\varphi ^h$ for any divisor $h$ of $n=2m+1$. 
  Since we have 
 $$4^m (F + \sqrt{-1} \psi) = ( \hat{F}^2 + \hat{\psi}^2 )^m 
 ( \hat{F} + \sqrt{-1} \hat{\psi} ) = 
  ( \hat{F} + \sqrt{-1} \hat{\psi} )^{m+1} ( \hat{F} - \sqrt{-1} \hat{\psi} )^m,$$
we conclude that,  if $F + \sqrt{-1} \psi =\varphi^h$, then
$$ 4^m  \varphi^h =  ( \hat{F} + \sqrt{-1} \hat{\psi} )^{m+1} 
( \hat{F} - \sqrt{-1} \hat{\psi} )^m.$$

We can for instance show that $h=1$ in the following situation:
if  $X'$ is factorial, any rational function $f$ can be written uniquely as $f = \frac{s}{t}$,
where $s,t$ are relatively prime sections of a line bundle $\sL$. We shall say that 
\begin{itemize}
\item
$f$ is irreducible if either $s$ or $t$ is prime, 
\item
$f$ is strongly  irreducible 
if both $s$ and $t$ are prime,
\item
$F = \frac{A}{B}$ is not associated to $f =\frac{s}{t}$
if $A$ is relatively prime to $s$ and 
$B$ is relatively prime to $t$.

\end{itemize}

Now, if   $( \hat{F} + \sqrt{-1} \hat{\psi} ) = \frac{s}{t} , \ 
( \hat{F} - \sqrt{-1} \hat{\psi} ) = \frac{A}{B}$ are irreducible and not associated, then 
write $\varphi =  \frac{a}{b}$. We obtain that 
$$ a^h = s^{m+1} A^m , b^h =  t^{m+1} B^m.$$
If $s$ is prime, then the first equality implies that $h$ divides $m+1$, and if $t$ is prime, then the second equality 
also implies $ h | (m+1)$. Similarly, if $A$ is prime, then $ h | m$,  and the same conclusion holds if $B$ is prime.
Hence $h$ divides $m, m+1$, hence $h=1$.

 \medskip

In the second case, where $n=2m+2$, we simply have to parametrize the quadric: setting  
$$ \phi = 1 + \la \tilde{\phi} , \hat{F} = 2, \hat{\psi} = \la \tilde{\psi} $$
we get 
$$ 8 \tilde{\phi} + 4 \la \tilde{\phi}^2 =  \la  \tilde{\psi}^2 \Leftrightarrow \la = \frac{8 \tilde{\phi} }{\tilde{\psi}^2 - 4  \tilde{\phi}^2} ,$$
hence the solutions correspond to the choice of two arbitrary functions $\tilde{\phi} , \tilde{\psi}$
 satisfying a similar condition to the  above one for the case where $n$ is odd.

Indeed in this case, up to constants, $$ F + \sqrt{-1} \psi = 
( 1 + \lambda \tilde{\phi} )^m ( 2 + \sqrt{-1} \lambda \tilde{\psi} ) .$$ 
A similar argument shows that, if  $1 + \lambda \tilde{\phi} , 2 + \sqrt{-1} \lambda \tilde{\psi} $
are irreducible and not associated, then $h=1 $.

\subsection{The case where $ \tau g \tau = g^{1 + n/2}$ (and $n$ is divisible by $8$)}The first basic observation here is that $G'$ contains the index $2$
subgroup  $G'^{ev} : = (2\ZZ / n \ZZ) \times (\ZZ/2) = : G^{ev} \times (\ZZ/2)$,  generated by $g^2$ and by $\tau$.

Hence we have a sequence of field extensions 
$$ \CC(X') \subset \CC(X^0) \subset \CC(X),$$

where $X^0 : = X / G^{ev}$. The extension $\CC(X^0) \subset \CC(X)$ is dihedral-like,
since $ \tau g^2 \tau = g^{2(1 + n/2)} = g^2$,
thus it has the description given in the previous subsection.

For convenience  set $ n = 2 N$, and  recall that $N$ is even.

Let  $$\CC(X) = \CC(X')[z] / (z^n - P), \ g(z) = \zeta z.$$

Then the extension $\CC(X^0) \subset \CC(X)$ is given as above by ($w : = \tau(z)$)
$$ (*) \ z w = \phi , \ z^N + w^N =  F = f + \tau (f), f^2 = P, $$
$$ (**) \  F^2 - 4 \phi^N = - \psi^2.$$ 

From $  g  \tau  = \tau g^{1 + n/2} $ follows that  
$$ g (w) = g  \tau (z) =  \tau g^{1 + n/2} (z) = \tau (- \zeta z ) = - \zeta^{-1} w,$$
hence $ g (\phi) = g ( zw) = - zw = - \phi,$
and $\phi^2$ is $G'$ invariant, so that $$\phi^2 = \Phi \in \RR(X').$$

Moreover $g (F) = - z^N - w^N = - F$; hence $\phi$ and $F$ generate the degree two extension 
 $ \RR(X') \subset \RR(X^0)$ and there exists therefore $A \in \RR(X') $ such that
 $ F = A \phi$.
We can rewrite now  condition $(**)$ as:
$$   \  A^2 - 4 \Phi^{N/2 -1}  = - (\psi )^2 / \Phi .$$
Observe that, since $ g (z^N ) = - z^N$, $\psi \in \RR(X^0)$ satisfies $g (\psi) = - \psi$,
and in particular $\psi^2 = \Psi \in \RR(X')$. 

Moreover, since $\phi, \psi $ generate the quadratic extension $ \RR(X') \subset \RR(X^0)$  there exists therefore 
$B \in \RR(X') $ such that
 $ \psi = B \phi$.
 
 Hence we can rewrite $(**)$ as
$$  (***) \  A^2 - 4 \Phi^{N/2 -1}  = - B^2.$$

Conversely, given $(***)$, we define $\psi, \phi$ by $\psi^2 = \Psi $, respectively $\phi^2 = \Phi$,
set $  F : = A \phi$ and observe that these generate a quadratic extension 
which we denote $ \RR(X') \subset \RR(X^0)$.

Then we use $(*)$ to get a cyclic extension $ \RR(X^0) \subset \RR(X)$.

Since $ g(F) = - F$, $g (\phi )= - \phi$, we can  extend $g$ to $\RR(X)$ setting
$ g (z) : = \zeta z$, $ g (w) : = - \zeta^{-1} w,$ and $\tau$ can be extended setting $ \tau (z) : = w$.

 The datum  of such functions $A, \Phi,  B$ satisfying $(***)$  is shown, exactly as in the previous subsection,
to be equivalent to the datum of two arbitrary functions in $\RR(X')$.

\subsection{The case where $ \tau g \tau = g^{-1 + n/2}$ (and $n$ is divisible by $8$)} Again  a basic observation here is that $G'$ contains the index $2$
subgroup $G'^{ev} : = (2\ZZ / n \ZZ) \rtimes (\ZZ/2)$, generated by $g^2$ and by $\tau$.

Hence we have a sequence of field extensions 
$$ \CC(X') \subset \CC(X^0) \subset \CC(X),$$

where $X^0 : = X / G^{ev}$. The extension $\CC(X^0) \subset \CC(X)$ is of standard type,
since $ \tau g^2 \tau = g^{2(-1 + n/2)} = g^{-2}$
so it has the description given in the first subsection (it is generated by a real element which is an eigenvector for $g^2$
with eigenvalue $\zeta^2$).

Set again $ n = 2 N$, where $N$ is even.

Pick as usual a generator $z$ of the extension such that  
 $$ z^n = z^{2N} = P \in \CC(X'), \ g(z) = \zeta z.$$

We have $ g^2 \tau (z) = \tau g^{-2} (z) = \zeta^2 \tau (z)$, and we set
$$ W : = z + \tau (z)  \Rightarrow W \in \RR(X).$$ 
If $ W =0$, then $ z = - \tau (z)$ and it follows that
$$ - \zeta z = g(-z) =  g \tau (z) = \tau g^{-1 + n/2} (z) = - \zeta \tau(z) = \zeta z,$$
a contradiction.

Hence $0 \neq W \in \RR(X)$ and it generates the cyclic extension $ \RR(X^0) \subset \RR(X),$
so that we may write $W^N = f \in  \RR(X^0) $.

We can write $ 2 z = W + u$, where $ u :  = z - \tau(z) \in \sqrt{-1} \RR(X).$ 

The action of $g$ is as follows:
$ g \tau (z) = \tau g^{-1 + n/2} (z) = - \zeta \tau(z)$, hence 
$$ g (W) = \zeta z  - \zeta \tau(z) = \zeta u, $$
$$ g(u) = g (z - \tau(z)) = \zeta ( z + \tau(z) ) = \zeta W.$$

Since $W^N = f \in \RR(X^0)$, $g (f) = g(W)^N = - u^N  \in \RR(X^0)$, hence $$u^N = h : = - g(f).$$ 

We have $ g(f) = -h, g(h) = -f$ hence 
$$y : = f + h  \Rightarrow y \in \RR(X^0)  , \ g (y) = -y  \Rightarrow y^2 = \Phi \in \RR(X').$$

Likewise $ g(fh) = fh \Rightarrow fh = \Psi \in \RR(X').$

Since every element  $f$  in $\RR(X^0)$ can be written in the form $ f = \al y  + \be , \ \al,\be \in \RR(X')$,
the equations $ f + h = y ,  g (f) = -h, f h = \Psi,$ imply that 
$$ f =  \frac{1}{2} y + \be, \ h = \frac{1}{2} y - \be, \ \Psi = \frac{1}{4} \Phi - \be^2.$$ 

Up to now we have described the field extension as a cyclic extension $W^N = f = \frac{1}{2} y + \be$ 
of a quadratic extension $y^2 = \Phi$, and clearly $ g(y) = -y$. 

 In this situation, the fact that the global extension is cyclic means
that   $g$ extends to the larger field;  indeed, we know that   $g( W) = \zeta u$, where $u^N = h =   \frac{1}{2} y - \be$,
hence the root $u$ must be in the field extension; and  since $W, u$ are both in the $\zeta^2$-eigenspace for $g^2$,
 this condition amounts to:
 $$ (a) \  \exists C, D \in \CC(X') \ {\rm such \ that} \ u = W (C + y D).   $$
 
 Recall that $\tau(W) = W, \tau (u) = -u$: hence 
 $$\tau (C + y D) = - C - y D \Rightarrow C, D \in \sqrt{-1} \RR(X').$$

Furthermore, 
$$ g (W) = \zeta u , \ g (u) = \zeta W \Rightarrow W = u g (C + y D) \Rightarrow $$
$$ (b) \ C^2 - \Phi D^2 = 1.$$

 Condition $ (b)$ means  that $(C, D, 1)$ is a point of the conic
$ X^2 - \Phi Y^2 - Z^2 = 0$.

 Since the real conic has already the rational point $( 1, 0 , 1)$, we get a  parametrization 
setting  $C = 1 + D \Theta$,  $\Theta \in \CC(X')$, and therefore, since $(1 + D \Theta)^2 - \Phi D^2 =1$,
we obtain 
$$ D = - \frac{2 \Theta}{\Theta^2 - \Phi}, \  C = - \frac{\Phi +  \Theta^2 }{\Theta^2 - \Phi}.$$

Multiplying numerator and denominator by $(\tau(\Theta)^2 - \Phi)$,
our condition is that 
$$ 2 \Theta ( \tau(\Theta)^2 - \Phi), \  (\Phi +  \Theta^2  ) ( \tau(\Theta)^2 - \Phi)$$
are both imaginary.

 An easy  solution is given by taking $\Phi =  \Theta \cdot \tau(\Theta)$.
 
 This is the only one, because, writing $ \Theta = A + \sqrt{-1} B$, the first condition is that
 $ (A+ \sqrt{-1} B) ( (A- \sqrt{-1} B)^2 - \Phi)$ is imaginary, equivalently,
 $$ A ( A^2 + B^2 - \Phi ) = 0.$$
 The solution $ A = 0$ must be discarded since then $C$ is real, a contradiction.
 
 Remains exactly the solution
 $$ A^2 + B^2  =  \Phi \Leftrightarrow \Phi =  \Theta \cdot \tau(\Theta),$$
 and then
 $$ (Theta) :  \   \Phi =  \Theta \cdot \tau(\Theta), \ D = - \frac{2 }{\Theta - \tau(\Theta)}, \  C = - \frac{\tau(\Theta) +  \Theta }{\Theta - \tau(\Theta)}.$$

The condition that $u$ is a root of $u^N = h = \frac{1}{2} y - \be$ is then expressed,  using (a), as
$$ (a') \   \frac{1}{2} y - \be = ( \frac{1}{2} y + \be)  (C + y D)^N,$$
equivalently it can be expressed by requiring that the ratio 
$$ \frac{  \frac{1}{2} y - \be  }  {  \frac{1}{2} y + \be} = \frac{1} {\Psi} (\frac{1}{2} y - \be)^2$$
is the $N$-th power of $(C + y D)$.

Observe that we wrote condition  $(a')$ as a condition in $\RR(X^0)$,
but this is indeed equivalent to two conditions in  $\RR(X')$, 
since we require that
$$(C + y D)^N = \sum_{j} {N\choose {2j}} C^{N-2j}  D^{2j} \Phi^j + y {N \choose {2j+1}} C^{N-2j-1}  D^{2j+1} \Phi^j  = : M + M' y$$ 
equals
$$ \frac{1} {\Psi} (\frac{1}{2} y - \be)^2 = \frac{1} {\Psi}  (\frac{1}{4} \Phi + \beta^2 ) - y \frac{1} {\Psi}   \beta.      $$

These equations can be rewritten   as two  equations for $\be$,
namely
$$M \Psi = M (\frac{1}{4} \Phi - \be^2) = \frac{1}{4} \Phi + \be^2, \ \  M' \Psi = M' (\frac{1}{4} \Phi - \be^2) = - \be,$$
hence, using the first equation to eliminate $\be^2$ in the second one,
$$ (a'') \ \  \be^2 (1 + M)  = \frac{1}{4} \Phi (M-1), \ \ $$
$$ (a''') \ \  -  \be = \frac{1}{2} \Phi \frac{M'}{M+1}.$$

Then $\be$ is determined by the  equation $ (a''')$,   and   $\be$ is real  since,  $C,D$ being imaginary,  $M, M'$ are real;
and we want  then $\be$ to yield a solution of the  equation $ (a'')$.  This amounts to:
$$(\frac{1}{2} \Phi \frac{M'}{M+1})^2  (1 + M)  = \frac{1}{4} \Phi (M-1) \Leftrightarrow \Phi  (M')^2= (M^2 -1).$$

However, this last condition $ 1 = (M^2 - \Phi  (M')^2) = M ^2 - y^2 (M')^2 $ is automatically true, since 
$$M ^2 - y^2 (M')^2 = (M + M' y) (M - M' y) = (C + y D)^N (C - y D)^N = (C^2 - \Phi D^2)^N = 1.$$ 

The final conclusion is that  also $\be$ is determined by  $\Theta \in  \CC(X')$.

Conversely, it  is now easy to see that we get a real cyclic covering with group $G = \ZZ/n = \ZZ/ 2N$, 
and  exhibiting 
the esoteric case  $ \tau g \tau = g^{-1 + n/2}$ provided that: 
\begin{itemize}
\item
 we are given $\Theta \in \CC(X'), $ such that 
 \item
 $\Phi : = \Theta \cdot \tau(\Theta)$ is not a square, and
 \item
 we are given  $C, D \in  \sqrt{-1} \RR(X')$    such that condition $ (b) \ C^2 - \Phi D^2 = 1$ holds (hence $C, D$ are determined by $\Theta$
as in  formula (Theta)).
\end{itemize}

In fact, letting $y$ be defined by $y^2 = \Phi$ and choosing (see remark \ref{realstructures}) the real structure 
on $\CC (X^0)$ such that $ \tau(y) = y$, and 
 letting  $\beta$ be  determined by $\Theta$ according to $(a''')$,  so that     
condition $(a')$ holds, 
the covering with group $G = \ZZ/n = \ZZ/ 2N$ is defined by the extension $$W^N = f : =   \frac{1}{2} y + \be .$$

 Defining then $\tau ( W) : = W, \ \tau (u) : = - u$ we obtain a real cyclic covering exhibiting 
the esoteric case  $ \tau g \tau = g^{-1 + n/2}$.

\bigskip
 
We shall  summarize the above discussion in the following theorem  \ref{field}:

\begin{theo}\label{field}
Let   $\RR(X), \RR(X'),$ be two real function fields, that is, two finitely generated field extensions of $\RR$ such that
$\CC(X) : = \RR(X) \otimes_{\RR} \CC$ and $ \CC(X')$ are fields with an antilinear automorphism ($\tau, \tau'$ respectively)
induced by complex conjugation on $\CC$.

Then a real cyclic covering $ \RR(X') \subset \RR(X)$, that is, an extension inducing a Galois extension 
$ \CC(X') \subset \CC(X)$ with Galois group $G = \langle g\rangle \cong  \ZZ/n$, and such that $\tau$ normalizes $G$ (then conjugation is given
via 
$\tau g \tau = g^m$, $m ^2=1 \in G$), is a fibre product of two such real cyclic coverings belonging to
the following four basic types:
\begin{enumerate}
\item
 $m=-1$ if and only if we are in the {\bf Standard (totally real) case}: there exist $f \in   \RR(X'), z \in \RR(X)$ such that
 $$ z^n = f ,  g (z) = \zeta z , \langle \zeta\rangle = \mu_n = \{ \eta | \eta^n = 1\}.$$ 
 \item
  $m= +1$ if and only if we are in the {\bf Dihedral-like  case}: there exist
  $$F, \phi, \psi \in \RR(X'), \ {\rm \ such \ that \ } (**) \   4 \phi^n = \psi^2 + F^2,$$
  such that the extension is given as
  $$ \CC(X) = \CC(X') (z) , \ {\rm \ and, \ for  \ } \ w: = \tau(z),  \ (*) zw = \phi, \ z^n + w^n = F.$$
  
  The choice of such $F, \phi, \psi \in \RR(X')$  satisfying $(**)$ is equivalent to the choice of two arbitrary functions
   $\tilde{\phi}, \tilde{\psi} \in \RR(X')$.
\item
$n$ is divisible by $8$ and, defining $N:=\frac{n}{2}$,  
$ m = 1 + N$ if and only if we are in the {\bf Twisted    case}:
there exist 
$$A, B, \Phi  \in  \RR(X'), \ {\rm \ such \ that \ } (***) \   A^2 + B^2 = 4 \Phi^{N/2 -1}$$
and such that  $\phi^2 = \Phi$ defines a real quadratic field extension $\RR(X') \subset \RR(X^0)$.
   
Then, setting $ F: = A \phi, \psi : = B \phi$, the field extension $\CC(X)$ is generated by $z$ such that ($w : = \tau(z)$)
$$ zw = \phi, \ z^N + w^N = F, \ g(z) = \zeta z, g(w) = - \zeta^{-1} w.$$
   
The choice of such functions $A, B, \Phi  \in  \RR(X')$  satisfying the above conditions  is equivalent to the choice of 
two arbitrary functions
in $\RR(X')$.
 
\item
$n$ is divisible by $8$ and, defining $N:=\frac{n}{2}$, 
 $ m = - 1 + N$ 
if and only if we are in the {\bf Esoteric   case}:
there exist 
$$\Phi , \be \in  \RR(X'),  \ C, D \in i \  \RR(X')  \ {\rm \ such \ that \ } (b) C^2 - \Phi D^2 =1,$$  and moreover,  
if  $y^2 = \Phi$ defines  a quadratic field extension $\RR(X') \subset \RR(X^0) \ni y$, we have:
$$ (a')  \frac{1}{2} y -  \be = ( \frac{1}{2} y + \be) (C + y D)^N.$$
Then, setting $f : = \frac{1}{2} y + \be$,  the real extension $\RR(X^0) \subset \RR(X)$ is generated by $W$, where   
$$ W^N = f , \ \tau(W) = W, g (W) = \zeta u , \ u : = W (C + y D), \ \tau(u) = -u.$$

The choice of such functions  satisfying the above conditions is equivalent to the choice of 
$\Theta = A + \sqrt{-1} B \in \CC(X')$ such that 
$ \Phi : = A^2 + B^2$ is not a square.

\end{enumerate}
\end{theo}

\section{Biregular structure of the coverings}

 The main purpose of this section is to use the description of  Abelian coverings through line bundles (invertible sheaves)
and effective (branch) divisors $D_g$.
  If $D$ is an effective divisor, then $\s (D)$ is also an effective divisor.

To see this, consider that, if $(X, \s)$ is a real variety, then  $\tau$ acts on $\CC(X)$, in particular it acts
on  the group of Cartier divisors  $ H^0(X, \CC(X)^* / \hol_X^*)$. 

If a collection $(U_{\alpha}, f_{\alpha}),$ (here $ \ f_{\alpha} \in \CC(X)^*$) defines 
a Cartier divisor $D$,  then the conjugate Cartier divisor is given by the collection $(\s(U_{\alpha}), \tau (f_{\alpha})).$ 

An invertible sheaf $\sL = \hol_X(D)$ has the property that its space of sections $H^0(X, \sL)$
is the space 
$$ \{ \phi \in \CC(X) | \phi f_{\alpha} = : \phi_{\alpha} \in \hol_X (U_{\alpha}) \} ,$$
hence $$\tau : H^0(X, \sL) \ra H^0(X, \tau (\sL)) = H^0(X, \hol_X( \s D)).$$

Given a cyclic covering of real varieties $X \ra X'$, we can replace $X'$ by a smooth real model $Y$, and replace $X$ 
by the normalization of the fibre product of the cyclic covering with the resolution $ Y \ra X'$, and obtain that (as in \cite{cyclic},
whose notation we shall now adopt)

\begin{enumerate}
\item
$Y = X / G$, $G \cong \ZZ/n$
\item
$X$ is normal, real, 
\item
$Y$ is real smooth, with $\s ' $ induced by $\s$;
\item
$ f : X \ra Y$ is finite and flat.

\end{enumerate}

More generally, we can  relax assumption (3) and assume that  $X, Y$ are  normal  real  projective varieties,
and $Y$ is moreover factorial.

 By flatness we have a decomposition 

$$ f_* ( \hol_X) = \hol_Y \oplus (\bigoplus_{\chi \in G^{\star}
\setminus \{0\}} \hol_Y (- L_{\chi})),$$
with a notation that we now explain.

 $\CC(X)$ is a cyclic  Galois extension of $\CC (Y)$, and we
denote now by $$G \cong \mu_n : = \{
\zeta \in \CC | \zeta^n = 1\}$$ its Galois group,   by $G^\star $ the
group of characters: $G^\star := Hom(G, \CC^\star)$, and we observe that  $G^\star  \cong \ZZ/n$, where 
the isomorphism $G^\star \overset{\cong}{\longrightarrow}\ZZ/n$ associates  to $\chi \in G^\star$
the residue class  $j\in \ZZ/n$ such that  $\chi(\zeta)=\zeta^j$, for all $\zeta \in  \mu_n$.

As  in \cite{cyclic} we observe that for each character $\chi$ of order  $n$ the extension is given by
$$\CC(X) =  \CC (Y)(w) , w^n = F(y) \in \CC (Y), $$ where $w$ is a
$\chi$-eigenvector.

Since $Y$ is  factorial,  $F$ admits a unique prime
factorization as a fraction of pairwise
relatively prime sections of line bundles, and we can write
$$  w^n = \frac{\prod_i \sigma_i^{n_i} } {\prod_j \tau_j^{m_j} } ,$$
 where the sections $\s_i, \tau_j$ are prime.

Writing  $$ n_i =  N_i + n n'_i , \ \  m_j = - M_j + n m'_j $$ with
$0 \leq N_i, M_j \leq n-1$ and setting
$$ z : = w \cdot  \prod_i \sigma_i^{- n'_i} \prod_j \tau_j^{m'_j},$$

we get 
a rational section $z$   of a line bundle on $Y$ and we have

$$   z^n = \prod_i \sigma_i^{N_i} \prod_j \tau_j^{M_j}  . $$

Put together  the prime factors which appear with the same
exponent,  obtaining:
$$   z^n = \prod_{i=1}^{n-1} \delta_i^{i}  . $$

Here each factor $\delta_j$ is reduced, but not irreducible, and
corresponds to a Cartier divisor that we
shall denote $D_j$. Calculating the  local monodromy around $D_j$ is easily seen that
 $D_j$ is exactly the divisorial part of the branch
locus $ D : =  \sum_j  D_j $ where the local
monodromy is the $j$-th power of the standard generator
$\ga : = e^{ \frac{2 \pi   \sqrt{-1} }{n}} $ of $G \cong \mu_n$.

We  write characters additively, in the sense
that we view them as
$ G^*
%{\star}  
= Hom (G, \ZZ / n)$. 

To
$\chi$ we associate  the normal covering
$$ Z_{\chi} : = X / ker (\chi). $$

A similar argument   (see for instance \cite{cyclic} page 285, section 1)  shows now that we have  a linear equivalence
   $$ (*)  \ n L _{\chi} \equiv \sum_i  [ \chi(i)] D_i $$
   where $[ r]$, for $r \in \ZZ / n$, is the unique residue class
   in $ \{ 0,1, \dots , n-1\}$.

   We observe for further use the following formula:
   $$  (I) \ \   [ \chi(i)]  +  [ \chi'(i)] =
[(\chi + \chi')(i)] + \epsilon^i_{\chi, \chi'}
n,$$
   which defines the numbers  $ \epsilon^i_{\chi, \chi'}  \in \{0,1\}$.

The following theorem is in part a special case of the structure
theorem for Abelian coverings
   due to  Pardini (\cite{pardini}, see also
\cite{comessatti}).

\begin{thm}\label{cyclic} i) Given a factorial variety $Y$, the
datum of a pair $(X, \gamma )$ where $X$ is
a normal variety and $\ga $ is an automorphism of order $n$ such
that, $G$ being the subgroup generated
by $\ga$, one has  $X / G \cong Y$, is equivalent to the datum of
reduced effective divisors $D_1, \dots , 
D_{n-1}$ without common components, and of a divisor class $L$ such
that we have the following linear
equivalence
$$ (*)  \ n L \equiv \sum_i i D_i $$ and moreover , setting $ h : =
G.C.D. \{ i | D_i \neq 0\} $, either

(**) $ h=1$ or  the divisor class
$$ (***)  \ L' : = \frac{n}{h} L -  \sum_i \frac{i}{h}  D_i $$ has
order precisely $h$.

ii) If $\LL$ is the geometric line bundle whose sheaf of
regular sections is $\hol_Y(L)$, then
$X$ is the normalization of the singular covering
$$ X' \subset  \LL, X' : = \{ (y,z)|  z^n = \prod_{i=1}^{n-1}
\delta_i^{i}  \}. $$ And $\ga$ acts by $ z
\mapsto e^{ \frac{2 \pi   \sqrt{-1} }{n}} z $.

iii) The scheme structure of $X$ is explicitly given as
$$ X : = Spec ( \hol_Y \oplus (\bigoplus_{\chi \in G^{\star}
\setminus \{0\}} \hol_Y (- L_{\chi}) ))$$ where
the divisor classes $L_{\chi}$ are recursively determined by $L_1 : =
L$, and by $ L_{\chi + \xi} \equiv
L_{\chi} + L_{\xi} - \sum_i  \epsilon^i_{\chi, \xi} D_i$.

Finally  the ring structure is given by  the multiplication maps
$$\hol_Y (- L_{\chi}) \times \hol_Y (-
L_{\xi}) \ra \hol_Y (- L_{\chi+ \xi})$$ determined by the section
$$\prod_i \delta_i^{ \epsilon^i_{\chi,
\xi}} \in  H^0 ( \hol_Y (- L_{\chi+ \xi} + L_{\chi} + L_{\xi})).$$

\end{thm}

\bigskip

\bigskip

At this stage we want to consider the extra data coming from the real structure.

It is convenient to view  $X$ embedded in $\bigoplus_{\chi \in
G^{\star} \setminus \{0\}} \LL_{\chi}$ and to 
write  the ring structure via the fundamental equation 
$$ (****) \ \  z_{\chi} \cdot z_{\xi}   =  z_{\chi+ \xi} \prod_i \delta_i^{
\epsilon^i_{\chi, \xi}},$$ 

where $z_{\chi} $ is
 a fibre variable on the geometric line bundle
$\LL_{\chi}$ (it is the natural section on $\LL_{\chi}$ of the pull back of  $ \hol_Y (- L_{\chi})$) 
and where we use the convention that  $z_0 =1$.

The action of $\tau$ on $\CC(X)$, hence also on the subfield $\CC(Y)$,  induces, by the description
that we recalled above of the building data of the cover,  an action of $\tau$ 
on the building data $L_{\chi}$, $D_i$, $z_\chi$, $\de_i$.

For simplicity of calculations, we write from now on the characters  as elements  
of $\ZZ/n$.

From the relations $\tau \gamma \tau = \gamma^m$, and $ \gamma z_j = \zeta^j z_j$, we obtain that
$$  \gamma ( \tau (z_j ) ) = \tau (\gamma ^m (z_j))=  \zeta ^{-mj} \tau ( z_j),$$
hence $\tau (z_j )$ is an eigenvector for $\zeta ^{-mj}$, and $\tau$  maps
$$ L_j \mapsto L_{-mj}.$$

Observe here that $ j = - m j $ if and only if $ j (m+1) = 0 \in \ZZ/n$; in particular, if $m=-1$,
all the line bundles are real and  linearized. The line bundles corresponding to $\{ j | j (m+1) \neq  0 \in \ZZ/n\}$
come in pairs, which are exchanged by $\tau$.

We look now at the branch divisors $D_i$. It is clear that the antiholomorphic
map $\s'$ carries the branch locus to itself, but we show now how the divisors $D_i$ are permuted.

We know that the local monodromy around $D_i$ is given by $\ga^i$, and we take local coordinates $(y, \de_i)$
($ y = y_1, \dots, y_{d-1})$, where $d = dim (Y)$, at the general point of $D_i$. Then, in appropriate local coordinates,
$\s'$ is given by $(\overline{y}, \overline{\de_i})$, and since $\s g^i \s = g^{mi}$, and the orientation in the normal bundle 
is reversed by $\s'$, it follows that 
$$ \s' (D_i) = D_{-mi},$$
in particular  
$\tau (\de_i ) =  \de_{-mi}$ and $D_i$ is real if and only if $i = - mi$.

Applying now $\tau$ to the fundamental equation (****) we obtain, since  $\tau (z_j) = z_{- mj}$, that 
the real structure extends to the covering if and only if 

$$    z_{-mj} \cdot z_{-mh}    =  z_{- m (j + h) } \prod_k \delta_{-mk}^{
\epsilon^k_{j, h }},$$ 
and since 
$$    z_{-mj} \cdot z_{-mh}    =  z_{- m (j + h) } \prod_i \delta_{i}^{
\epsilon^i_{-mj, -mh }},$$
this is equivalent to requiring that
$$ \prod_i \delta_i^{
\epsilon^i_{- mj, - m h}} = \prod_i \delta_{i}^{
\epsilon^{-mi}_{j, h}},$$
a property which clearly holds true by virtue of (I) (just observe that $ -m j (i) = j (-mi) , -m h (i) = h (-m i)$).

\medskip

It is now easy to derive the real version of the biregular theory of cyclic coverings.

\begin{theo}\label{real-cyclic} i) Given a real factorial variety $(Y, \tau_Y)$, the
datum of a real finite cyclic covering $(X, \tau_X) \ra (Y, \tau_Y)$,
where $X$ is
a normal variety and the  group $G \cong \ZZ/n$  with $X/G  \cong Y$ is generated  by an automorphism
$\ga $ of order $n$, 
 is equivalent to:
 \begin{enumerate}
 \item
 the datum of an element $m \in \ZZ/n$ with $m^2=1$,
  \item
  the datum of
reduced effective divisors $D_1 = div (\de_1), \dots , D_{n-1} = div(\de_{n-1})$ without common components, and such that 
$$
\tau_Y (\de_i) = \de_{-mi}, \quad 
{\rm for}\quad i=1, \ldots , n-1,
$$
\item
and the datum of a divisor class $L $ such
that we have the following linear
equivalence
$$ (*)  \ n L \equiv \sum_i i D_i .$$
The above data should satisfy the following conditions:

\begin{itemize}
\item[(I)]  setting $ h : = G.C.D. \{ i | D_i \neq 0\} $, either (**) $ h=1$ or, setting $ n = hd$, the divisor class
$$ (***)  \ L' : = \frac{n}{h} L -  \sum_i \frac{i}{h}  D_i $$ has
order precisely $h$ (this condition guarantees that $Y$ is an  irreducible  variety, otherwise it is just a normal  scheme);
\item[(II)]
defining  divisor classes $L_{\chi}$  recursively  by $L_0 = \hol_Y, L_1 : =
L$, and by $ L_{\chi + \xi} \equiv
L_{\chi} + L_{\xi}  -  \sum_i  \epsilon^i_{\chi, \xi} D_i$,  
  there are  choices of  tautological sections
$z_j $ on $\LL_j$ of the pull back of  $ \hol_Y (- L_j)$ such that 
$$ \tau_Y ( z_j) = z_{-mj},$$
in particular $\tau (L_j) \cong L_{-mj}$. 
\end{itemize}
\end{enumerate}

ii) In the situation of i)
$X$ is the normalization of the singular covering
$$ X' \subset  \LL, X' : = \{ (y,z)|  z^n = \prod_{i=1}^{n-1}
\delta_i^{i}  \}, $$ 

and $\ga$ acts by $ z
\mapsto e^{ \frac{2 \pi   \sqrt{-1} }{n}} z $.

iii)  The scheme structure of $X$ is explicitly given as
$$ X : = Spec ( \hol_Y \oplus (\bigoplus_{\chi \in G^{\star}
\setminus \{0\}} \hol_Y (- L_{\chi}) )),$$
that is, $X$ is embedded in $\bigoplus_{\chi \in
G^{\star} \setminus \{0\}} \LL_{\chi}$ and defined  
by the fundamental equations 
$$ 
(****) \ \  z_{\chi} \cdot z_{\xi}   =  z_{\chi+ \xi} \prod_i \delta_i^{\epsilon^i_{\chi, \xi}}. 
$$ 
 
iv) The real structure $\tau_X$ on $X$ is defined by extending $\tau_Y$ from  $\hol_Y$ to $\hol_X$
 via the action of $\tau_Y$ on the $z_j, \de_i$'s (clearly then $\tau g \tau = g^m$ for all $g \in G$).
\end{theo}

\begin{rem}\label{reality}
(a) Since $Y$ is a complete variety,  condition (I) 
in Theorem \ref{real-cyclic}, ensuring connectedness of the covering,
may be replaced by $H^0 (  \hol_Y (- L_{\chi}) )= 0$
for all $\chi \neq 0$.

(b) Condition (II) in Theorem \ref{real-cyclic} 
holds if it holds for $j=1$,  as it follows inductively by  the linear 
equivalences (II).

In order to spell out concretely 
condition (II) for $j=1$, $\tau (L_1) \cong L_{-m}$, let us set $\mu = [ -m]$,
and use the recursive definition of the $L_\chi$'s, 
for convenience working in the Picard group of $Y$.

 Then 
$$ L_{-m} = L_{\mu} =  \mu L_1 - \sum_i (\epsilon^i_{1,1} + \epsilon^i_{2,1} + \dots + \epsilon^i_{\mu - 1,1} ) D_i.$$

 Now, because of the equation 
$$ n  L_{\mu} =  \mu \sum_i i D_i  -  \sum_i n (\epsilon^i_{1,1} + \epsilon^i_{2,1} + \dots + \epsilon^i_{\mu - 1,1} ) D_i = \sum [\mu i] D_i,$$

$$\eta : = \tau (L_1) -  L_{\mu} $$ 
satisfies $$ n \eta =  \tau (\sum_i i D_i) - \sum_i  [\mu i] D_i = \sum_i (i  - [ i] ) D_{-mi} = 0,$$

hence $\eta$ 
 is  an  $n$-torsion divisor (a torsion divisor of order dividing $n$). If $\eta \neq 0$, one can still alter the choice of $L_1$  keeping the divisors $D_i$ fixed, in view of (3),
adding another $n$-torsion divisor, call it $\la$.

We can then satisfy condition (II) for $j=1$ if we find $\la$ which solves the equation 
$$ \eta = \tau (\la) - \mu \la.$$

For $m=-1$, $\mu=1$, where condition (II)
for $j=1$ simply means that $L_1$ should be real,
 we want to find $\la$  solving:
$$ \eta : = \tau (L_1) -  L_1 = \tau (\la) -  \la.$$

\end{rem}

\begin{rem}
A quite analogous theorem holds, mutatis mutandis,   for real Abelian coverings, describing real  normal (not necessarily connected) schemes.

These are the changes to be done:

\begin{itemize}
\item
$\ZZ/n$ is replaced by a finite Abelian group $G$.
\item
in (1), $m$ is replaced by $ M \in Aut(G)$ such that $M^2=1$.
\item
in (2), we choose divisors $D_g$ for $ g \in G, \ g \neq 0$.
\item
(3) and (I) disappear, while (II) 
is replaced by : the character sheaves $ L_{\chi}$ 
must satisfy $$ L_{\chi + \xi} \equiv
L_{\chi} + L_{\xi} - \sum_i  \epsilon^i_{\chi, \xi} D_i$$ (see \cite{pardini})
\item
in (II) we replace $z_j$ by $z_{\chi}$, $L_j$ by $ L_{\chi}$ , and so on..
\item
ii)  disappears, iii) and iv) are identical.
\item
If $Y$ is complete, then connectedness of $X$ is verified by imposing that the global sections of $\hol_X$
are just the constants.
\end{itemize}

\end{rem}

\end{document}